\documentclass[lettersize,journal]{IEEEtran}
\usepackage{amssymb}
\usepackage{amsmath,amsfonts}
\usepackage{algorithm}
\usepackage{algpseudocode}
\usepackage{array}
\usepackage{textcomp}
\usepackage{stfloats}
\usepackage{url}
\usepackage{verbatim}
\usepackage{multirow}
\usepackage{graphicx}
\usepackage{subfigure}
\usepackage{cite}
\usepackage{amsthm}
\usepackage{booktabs}
\usepackage{makecell}

\usepackage[T1]{fontenc}
\usepackage{aecompl}

\usepackage{cases}
\usepackage{soul}
\usepackage{subfigure}
\usepackage{url}
\usepackage{verbatim}
\usepackage{siunitx}
\usepackage{graphicx}

\newtheorem{lemma}{Lemma}

\newtheorem{theorem}{Theorem}

\newtheorem{remark}{Remark}

\begin{document}

\title{Causality-Informed Data-Driven Predictive Control}

\author{Malika Sader, Yibo Wang, Dexian Huang, Chao Shang, \IEEEmembership{Member, IEEE}, and Biao Huang, \IEEEmembership{Fellow, IEEE}
\thanks{This work was supported by National Natural Science Foundation of China under Grants 62373211 and 62327807. \textit{(Corresponding author: Chao Shang.)}}
\thanks{M. Sader and Y. Wang are with Department of Automation, Tsinghua University, Beijing 100084, China (e-mail: mlksdr@mail.tsinghua.edu.cn, wyb21@mail.tsinghua.edu.cn). }

\thanks{D. Huang and C. Shang are with Department of Automation, Beijing National Research Center for Information Science and Technology, Tsinghua University, Beijing 100084, China (e-mail: huangdx@tsinghua.edu.cn, c-shang@tsinghua.edu.cn). }

\thanks{B. Huang is with Department of Chemical and Materials Engineering, University of Alberta, Edmonton, AB T6G 1H9, Canada (e-mail: biao.huang@ualberta.ca). }}

\markboth{IEEE Transactions on Control Systems Technology}%
{Sader \MakeLowercase{\textit{et al.}}: Causality-Informed Data-Driven Predictive Control}

\maketitle

\begin{abstract}
 As a useful and efficient alternative to generic model-based control scheme, data-driven predictive control (DDPC) is subject to bias-variance trade-off and is known to not perform desirably in face of uncertainty. Through the connection between direct data-driven control and subspace predictive control, we gain insight into the reason being the lack of causality as a main cause for their high variance of implicit prediction. In this \textcolor{black}{brief}, we derive \textcolor{black}{a new causality-informed formulation of DDPC} as well as its regularized form that balances between control cost minimization and implicit identification of a causal multi-step predictor. Since the proposed causality-informed formulations only call for block-triangularization of a submatrix in the generic non-causal DDPC based on LQ factorization, our causality-informed formulation of DDPC enjoys computational efficiency. Its efficacy is investigated through numerical examples and application to model-free control of a simulated industrial heating furnace. Empirical results corroborate that the proposed method yields obvious performance improvement over existing formulations in handling stochastic noise and process nonlinearity.

\end{abstract}


\section{Introduction}
\label{sec:introduction}
\IEEEPARstart{M}{odel} predictive control (MPC) has been broadly employed in industrial practice due to its efficacy to tackle constraints and multiple performance criteria. To achieve desirable performance, a sufficiently accurate process model is indispensable in MPC, which requires either first-principle modeling or identification from data \cite{hjalmarsson1996model,pan2022stochastic}. The recent years have witnessed a paradigm shift from model-based control towards the so-called data-driven predictive control (DDPC) \cite{breschi2023data,Markovsky2021,verheijen2023handbook}, whose idea can be traced back to the early subspace predictive control (SPC) \cite{favoreel1999spc}. Therein, a non-parametric predictor is constructed from an input/output trajectory while bypassing identification exercise that is possibly cumbersome. Due to such intriguing features in a data-rich world, DDPC has found widespread applications in power systems, motion control, smart buildings, fuel cell systems, to name a few. See \cite{Markovsky2021} for a comprehensive review. 

Building upon behavioral systems theory \cite{willems2007behavioral}, data-driven control design has pros and cons of its own. Like two sides of a coin, the direct DDPC scheme excels in face of ``bias" error induced by inadequate model fitting, e.g., nonlinear systems \cite{martin2023guarantees}, but may not perform well in the case of ``variance" uncertainty \cite{krishnan2021direct,fiedler2021relationship}. In the latter case, the indirect method is typically superior to DDPC owing to the denoising effect of explicit model identification. Such a bias-variance trade-off has been theoretically formalized and investigated in \cite{dorfler2022bridging,krishnan2021direct}. In \cite{dorfler2022bridging}, an intermediate between two extremal formulations was achieved by endowing DDPC with a novel regularizer. 

For performance improvement, extensive efforts have been made in integrating useful prior knowledge into data-driven control design \cite{berberich2022combining}. As is known to all, the multi-step predictor embedded in predictive control shall be causal \textit{a priori}, which is naturally ensured by parametric models such as a state-space representation. The causality of predictive models has been intensively studied in subspace identification literature \cite{Qin2005,m,peternell1996statistical}, while the lack of causality was identified as a main cause for inflated estimation errors. In data-driven control design, it is also important to ensure strict causality of data-driven non-parametric representations of linear time-invariant (LTI) systems. In a recent work \cite{o2022data}, a modified formulation of DDPC was put forth by segmenting trajectories within the control horizon. This can partially alleviate (but not thoroughly remove) the non-causal effect in predictions, which, however, comes at a higher computational cost due to additional decision variables involved in the modification.

These relevant endeavors indicate that the lack of causality in DDPC is a critical factor to its underperformance in face of uncertainty. In this \textcolor{black}{brief}, we investigate \textit{causality-informed} realizations of DDPC, in order to boost its practical performance and narrow the gap with model-based schemes under stochastic uncertainty. \textcolor{black}{The primary tool we utilize is the LQ factorization, which plays a key role in $\gamma$-DDPC \cite{breschi2023data}, a useful realization of DDPC with reduced complexity and clear interpretability.} \textcolor{black}{We first revisit the relation between $\gamma$-DDPC and SPC, discuss the equivalence between their regularizations, and reveal their lack of causality}. Based on this, we develop a simple yet effective causal formulation of $\gamma$-DDPC, which only calls for block-triangularizing a submatrix in generic non-causal $\gamma$-DDPC. Moreover, a regularized form is further developed that extends the \textcolor{black}{known regularizations of non-causal $\gamma$-DDPC in \cite{breschi2023data,breschi2023uncertaintyaware,breschi2023impact}} to a causality-informed setup. We investigate the efficacy of the causality-informed DDPC using numerical examples and high-fidelity simulations of an industrial heating furnace system. Empirical results showcase that, the enforced causality in DDPC helps to suppress output prediction errors when there exists large variance uncertainty or model mismatch caused by system nonlinearity, thereby always leading to control performance improvement over non-causal DDPC schemes. 
It is worth noting that even in the presence of nonlinearity, enhanced prediction accuracy and control performance can be attained by enforcing causality.


\textbf{Notation:} We denote by $\mathbb{Z}$ ($\mathbb{Z}^+$) the set of (positive) integers. 
For a matrix $X$, $\| X \|_F$ denotes the Frobenius-norm and $X^\dagger$ denotes the Moore-Penrose inverse. $X(i:j)$ denote the submatrix constructed with the elements from the $i$-th to the $j$-th row of $X$. For an $L\times L$ block matrix $X \in \mathbb{R}^{pL \times mL}$ where each block has size $p\times m$, the operator $\mathcal{LT}_{p,m}(X)$ returns the lower-block triangular part of $X$. \textcolor{black}{We define $\mathbf{col}(x(i),x(i+1),\cdots,x(j))=[x(i)^\top~x(i+1)^\top~\cdots~x(j)^\top]^\top$ and $\mathbf{col}(X_1,X_2,\cdots,X_n) = [X_1^\top~X_2^\top~\cdots~X_n^\top]^\top$. Given a sequence $\{ x(i) \}_{i=1}^N\in\mathbb{R}^n$, $x_{[i:j]} = \mathbf{col}(x(i),x(i+1),\cdots,x(j))$ denotes the restriction of $x$ to the interval $[i,j]$.
} A block Hankel matrix of depth $s$ can be constructed from ${x}_{[i:j]}$ via the following defined block Hankel matrix operator $\mathcal{H}_s(x_{[i:j]})$. A sequence $x_{[i:j]}$ is said to be persistently exciting of order $s$ if the Hankel matrix $\mathcal{H}_s(x_{[i:j]})$ has full row rank. \textcolor{black}{Given an $n$th-order state-space model $(A,B,C,D)$, its lag is denoted as $\ell(A,B,C,D)$, i.e., the smallest integer $l$ such that $\mathbf{col}(C,CA,...,CA^{l-1})$ has rank $n$.}
 
\section{Preliminaries}\label{section 2}
\subsection{Basics of DDPC}
Consider the following stochastic discrete-time LTI system:
\begin{equation}
    \label{equation: LTI system}
    \left \{    \begin{aligned}
    x(t+1)&=Ax(t)+Bu(t)+w(t)\\
    y(t)&=Cx(t)+Du(t)+v(t)
    \end{aligned} \right .
\end{equation}
where ${x}(t)\in\mathbb{R}^n$, ${u}(t)\in\mathbb{R}^m$ and ${y}(t)\in\mathbb{R}^p$ stand for state, input and output, respectively. $w(t) \in\mathbb{R}^m$ and $v(t) \in\mathbb{R}^p$ denote process and measurement noise, respectively. It is assumed that the system \eqref{equation: LTI system} is minimal. \textcolor{black}{In a data-driven environment, system matrices $(A,B,C,D)$ are not known but an input-output trajectory $\{u_d(i), y_d(i)\}_{i=1}^{N_d}$ can be collected offline, where $\{u_d(i)\}_{i=1}^{N_d}$ is persistently exciting of order $L+n$ and $N_d \ge (m+1)L+n+1$ with $L=L_{\rm p}+L_{\rm f}$ ($L_p$ and $L_{\rm f}$ denote past horizon and future horizon, respectively). We first recall the following result on data-driven representation of deterministic LTI systems.}


\begin{theorem} \label{Fundamental Lemma}
(\cite{markovsky2008data}) For system \eqref{equation: LTI system} with $w(t)=0$ and $v(t)=0$, $U_d = \mathcal{H}_L({u}_{d,[1,N]})$ and $Y_d = \mathcal{H}_L({y}_{d,[1,N]})$ are block Hankel matrices of inputs and outputs. Consider the past input-output data $u_{\rm p}= u_{[t-L_{\rm p}:t-1]}$ and $y_{\rm p}= y_{[t-L_{\rm p}:t-1]}$ and the future input $u_{\rm f}= u_{[t:t+L_{\rm f}-1]}$. Under the condition $L_{\rm p}\ge\ell(A,B,C,D)$, the future output $y_{\rm f}=y_{[t:t+L_{\rm f}-1]}$ can be uniquely decided by $Y_{\rm f}g=y_{\rm f}$, where $g \in \mathbb{R}^{N_d - L + 1}$ is an additional variable solving the following equations:
        \begin{equation}
            \label{equation: DeePC division}
                \mathbf{col}(Z_{\rm p},U_{\rm f})g=
                \mathbf{col}(z_{\rm p},u_{\rm f})
        \end{equation}
        with $Z_p = [U_p^\top~~ Y_p^\top]^\top$ and $z_p = [u_p^\top~~ y_p^\top]^\top$, where $U_{\rm p} = U_{{\rm d},[1:mL_{\rm p}]}\in\mathbb{R}^{mL_{\rm p}\times (N_d-L+1)},~
                U_{\rm f} = U_{{\rm d},[mL_{\rm p}+1:mL]}\in\mathbb{R}^{mL_{\rm f}\times(N_d-L+1)}$, 
        and $Y_{\rm p},Y_{\rm f}$ are constructed in a similar way. 
\label{thm: fundamental lemma}
\end{theorem}
Owing to Theorem \ref{Fundamental Lemma}, the optimization problem of DDPC can be formulated as \cite{Coulson2019}:
\begin{subequations}
        \begin{align}
        \min_{u_f,\hat y_f,g}\quad&\mathcal{J}(u_f,\hat y_f) \\
        \mathrm{s.t.}\quad~&
        \label{eq: aDeePC0}
        {Y}_fg=\hat y_f,~\mbox{Eq.}~\eqref{equation: DeePC division}\\
        & u_f\in\mathbb{U},\ \hat y_f\in\mathbb{Y},
        \label{eq: bDeePC0}
        \end{align}
        \label{eq: DeePC0}
\end{subequations}
where $\mathbb{U}$ and $\mathbb{Y}$ are input and output constraint sets. The objective of \eqref{eq: DeePC0} can be specified as $\mathcal{J}(u_f,\hat{y}_f) = ||\hat{y}_f-r||_{\mathcal{Q}}^2+||{u}_f||_{\mathcal{R}}^2$,
where $r$ is a reference signal, and $\mathcal{R},\mathcal{Q} \succ 0$ are input and output cost matrices. 

In the presence of uncertainty, i.e. $w(t) \neq 0$ and $v(t) \neq 0$, the solution of $g$ in \eqref{eq: DeePC0} may become ill-conditioned. A useful remedy is to adopt the Moore-Penrose inverse, i.e.,
\begin{equation}
    g^*_{\rm pinv} = \begin{bmatrix}Z_p\\ {U}_f\end{bmatrix}^\dagger \begin{bmatrix}z_p \\  u_f \end{bmatrix},
    \label{eq: gpinv}
\end{equation}
which is the minimum-norm solution to \eqref{equation: DeePC division}. This supplies a direct data-driven predictor of $\hat{y}_f$ using
$z_p$ and $u_f$, based on which the problem of SPC is formulated as \cite{favoreel1999spc}:
 \begin{subequations}
        \label{eq: DeePC0p}
        \begin{align}
        \min_{u_f,\hat y_f}\quad&\mathcal{J}(u_f,\hat y_f) \\
        \mathrm{s.t.}\quad~&
        \label{eq: aDeePC0p}
        \hat{y}_f = Y_f   g^*_{\rm pinv},~\mbox{Eq.}~\eqref{eq: gpinv}\\
        & u_f\in\mathbb{U},\ \hat y_f\in\mathbb{Y}.
        \label{eq: bDeePC0p}
        \end{align}
    \end{subequations}
 


\subsection{$\gamma$-DDPC via LQ factorization}
In \cite{breschi2023data}, \textcolor{black}{$\gamma$-DDPC} was derived by performing LQ factorization of the joint input-output block Hankel matrix:
\begin{equation}
\begin{bmatrix}
    Z_p \\ U_f \\ Y_f
\end{bmatrix} = \begin{bmatrix}
L_{11} & 0 & 0 \\
L_{21} & L_{22} & 0 \\
L_{31} & L_{32} & L_{33} \\
\end{bmatrix}\begin{bmatrix}
Q_1 \\
Q_2 \\
Q_3
\end{bmatrix},
\label{eq: LQ factorization}
\end{equation}
where matrices $\{ L_{ii} \}_{i=1}^3$ are non-singular and $\{ Q_i \}_{i=1}^3$ are orthonormal matrices satisfying $Q_iQ_i^\top = I$ and $Q_iQ_j^\top = 0,~ \forall i \neq j$. Combining \eqref{eq: aDeePC0} with \eqref{eq: LQ factorization}, we arrive at:
\begin{equation}
    \begin{bmatrix} z_p \\ u_f\\\hat y_f \end{bmatrix}\\=\begin{bmatrix}
    Z_p \\ U_f \\ Y_f
\end{bmatrix} g= \begin{bmatrix}
L_{11} & 0 & 0 \\
L_{21} & L_{22} & 0 \\
L_{31} & L_{32} & L_{33} \\
\end{bmatrix}\begin{bmatrix}
Q_1 \\
Q_2 \\
Q_3
\end{bmatrix}g.
\label{eq: LQ factorization-relationship}
\end{equation}
By defining a new decision variable $\gamma=\begin{bmatrix} 
\gamma^{\top}_1&\gamma^{\top}_2&\gamma^{\top}_3
\end{bmatrix}^{\top} = [Q_1 g; Q_2 g; Q_3 g]$, the data-driven control design problem \eqref{eq: DeePC0} and its regularized version can be recast as follows, known as $\gamma$-DDPC
\cite{breschi2023uncertaintyaware,breschi2023impact}:
\begin{subequations}
    \label{eq: r-gamma DeePC}   
    \begin{align} \label{eq: 0r-gamma DeePC} 
        \min_{u_f,\hat y_f,\gamma} &~ \mathcal{J}(u_f,\hat  y_f) +\mu  \cdot \Phi(\gamma)  \\
        \label{eq: r-gamma DeePC1} 
        {\rm s.t.}~~ &~ \begin{bmatrix}
          L_{11} & 0 &0\\
            L_{21} & L_{22} &0\\
            L_{31} & L_{32}&L_{33}
        \end{bmatrix} \begin{bmatrix} 
       \gamma_1 \\ \gamma_2\\ \gamma_3
        \end{bmatrix} = \begin{bmatrix} 
        z_p\\u_f \\ \hat y_f
        \end{bmatrix} ,\\
        \label{eq: r-gamma DeePC2} 
        &~ u_f \in \mathbb{U},~ \hat y_f \in \mathbb{Y},
      \end{align}
 \end{subequations}
    where $\Phi(\gamma)$ and $\mu > 0$ stand for the regularizer and the weight, respectively. \textcolor{black}{By using a finite $\mu$, the input-output predictor underlying \eqref{eq: r-gamma DeePC}, which is not explicitly identified, does not adhere to the SPC predictor \eqref{eq: aDeePC0p} anymore but can be time-varying when \eqref{eq: r-gamma DeePC} is resolved iteratively in a receding horizon fashion.} This enables to flexibly balance between minimizing the control cost and identifying a predictor from data, which may yield improved performance \cite{krishnan2021direct}. Besides, by enforcing $\gamma_3 = 0$, $\gamma$-DDPC becomes equivalent to the generic SPC \cite{breschi2023data}. 

    \begin{theorem} [Equivalence to SPC \cite{breschi2023data}] \label{thm: standard SPC}
Solving the $\gamma$-DDPC problem \eqref{eq: r-gamma DeePC} with the norm-based regularizer $\Phi(\gamma) = \| \gamma_3 \|^2$ and $\mu \to \infty$ yields the same control decision $u_f^*$ and output prediction $\hat{y}_f^*$ as those by solving \eqref{eq: DeePC0p}. 
 \end{theorem}
    

In \cite{breschi2023uncertaintyaware,breschi2023impact}, it was pointed out that regularizing $\gamma$ is related to regularizing $g$ in \eqref{eq: DeePC0}. For the latter, a notable projection-based regularizer was coined in \cite{dorfler2022bridging}, based on which the regularized DDPC (R-DDPC) problem is formulated as:
\begin{subequations}
        \label{eq: r-DeePC00}
        \begin{align}
        \min_{u_f,\hat  y_f,g}\quad&\mathcal{J}(u_f,\hat  y_f)+\mu \cdot \|(I-\Pi)g \|^2 \label{eq: rg-DeePC001}\\
        \mathrm{s.t.}\quad~&
        \label{eq: r-DeePC001}
        \begin{bmatrix} Z_p \\ U_f  \\ {Y}_f\end{bmatrix} g = 
    \begin{bmatrix} z_p \\ u_f\\\hat{y}_f \end{bmatrix},\\
        & u_f\in\mathbb{U},\ \hat{y}_f\in\mathbb{Y},
        \label{eq: r-DeePC002}
        \end{align}
    \end{subequations}
    where 
    $\Pi=\begin{bmatrix} Z_p \\ U_f\end{bmatrix}^{\dagger}\begin{bmatrix} Z_p \\ U_f\end{bmatrix}$
is a orthogonal projection matrix onto the column span of $[Z_p^\top~ U_f^\top]^\top$. Next we unveil the connection between \eqref{eq: r-gamma DeePC} and  \eqref{eq: r-DeePC00} for a suitable choice of $\Phi(\gamma)$.  

\begin{lemma}\label{corally1}
Solving the regularized problem \eqref{eq: r-gamma DeePC} with the norm-based regularizer $\Phi(\gamma)= \| \gamma_3\|^2$ yields the same control decision $u_f^*$ and output prediction $\hat{y}_f^*$ as those by solving the R-DDPC problem  \eqref{eq: r-DeePC00}.
\end{lemma}

The proof of Lemma \ref{corally1} can be made by borrowing ideas from \cite{breschi2023data} and is presented in \textcolor{black}{Supplementary Information due to page limitation}. Indeed, Lemma \ref{corally1} extends Theorem \ref{thm: standard SPC} to cases with \textcolor{black}{finite} $\mu$. This provides a clear vision on the effect of regularizing $\gamma_3$, which in a similar spirit as the projection-based regularizer $\|(I-\Pi)g \|^2$ by trading-off implicit model identification against control cost minimization.
Thus, problem \eqref{eq: r-gamma DeePC} offers a unified regime of DDPC strategies at large, which enjoys interpretability, computational efficiency, and compatibility with regularizers. However, the lack of causality remains a critical issue. Next, building upon \eqref{eq: r-gamma DeePC}, we develop a \textcolor{black}{causal $\gamma$-DDPC and its regularized form}, which only ask for a slight modification of generic $\gamma$-DDPC and thus inherits above merits.

\section{Causality-Informed Data-Driven Predictive Control Formulation}

\subsection{\textcolor{black}{Causal form of SPC} }
Indeed, \eqref{eq: aDeePC0p} is the SPC predictor originally proposed in \cite{favoreel1999spc} and plays an equally critical role in subspace identification. To see this, we construct the following linear multi-step predictor:
\begin{equation}
    y_f = \underbrace{[K_p ~~ | ~~ K_f]}_{\triangleq K} \cdot \begin{bmatrix}
        z_p \\ u_f
    \end{bmatrix},
    \label{eq:multistep predictor}
\end{equation}
where $K_p$ and $K_f$ are associated with initial input/output trajectory and future inputs, respectively. An ordinary least-squares fitting can be made based on Hankel matrix data:
\begin{equation}
    \min_{K}~ \left \| Y_f - K \cdot \begin{bmatrix}
        Z_p \\ U_f
    \end{bmatrix} \right \|_F^2.
    \label{eq: LS fit}
\end{equation}
Without making any structural presumption on $K$, solving \eqref{eq: LS fit} yields the classic SPC predictor \cite{favoreel1999spc}, as already used in \eqref{eq: aDeePC0p}:
\begin{equation}
    K^* = Y_f \begin{bmatrix}
        Z_p \\ U_f
    \end{bmatrix}^\dagger,~\hat{y}_f = K^* \begin{bmatrix} z_p \\ u_f \end{bmatrix} = Y_f \begin{bmatrix}Z_p\\ {U}_f\end{bmatrix}^\dagger \begin{bmatrix} z_p \\ u_f \end{bmatrix}.
    \label{eq: SPC predictor}
\end{equation}
The SPC predictor has mostly emerged as an intermediate in subspace identification. In principle, $K_f$ is supposed be block lower-triangular in order to enforce strict causality, a fundamental attribute that shall be possessed by the multi-step predictor \eqref{eq:multistep predictor}. It was shown that the non-causal terms in $K_f$ will asymptotically vanish with $N_d \to \infty$. However, in a finite sample regime, the least-squares fit in \eqref{eq: LS fit} inevitably leads to a non-causal relation between $U_f$ and $Y_f$ \cite{m}, and consequently, increased error variances and risk of model over-fitting \cite{Qin2005}. To preserve causality of predictor, one may turn to a constrained least-squares formulation:
\begin{equation}
\begin{split}
    \min_{K}~& \left \| Y_f - K \cdot \begin{bmatrix}
        Z_p \\ U_f
    \end{bmatrix} \right \|_F^2 \\
    {\rm s.t.}~~& ~K = [K_p ~~ K_f],~ K_f\text{ is~lower-block~triangular}
\end{split}    
\label{eq: causal K}
\end{equation}
which has been studied in subspace identification \cite{Qin2005,m,peternell1996statistical}.

Our focus is then shifted towards data-driven optimal control tasks. By embedding the data-driven causality-informed solution \eqref{eq: causal K} of SPC into control design, we arrive at the following formulation called causal SPC (C-SPC), which is a bilevel program: 
\begin{subequations}
    \begin{align}
        \min_{u_f,\hat{y}_f} &~ \mathcal{J}(u_f,\hat{y}_f) \\
        {\rm s.t.}~ &~ \hat{y}_f = K^* \begin{bmatrix}
            z_p \\
            u_f
        \end{bmatrix},~ K^*~{\rm solves~\eqref{eq: causal K}}, \\
        &~ u_f \in \mathbb{U},~ \hat{y}_f \in \mathbb{Y}.
    \end{align}
    \label{eq: multi-step predictive control}
\end{subequations}
As compared to the generic SPC in \eqref{eq: DeePC0p}, C-SPC encodes prior knowledge of causality within $K^*$, which is expected to improve prediction accuracy and control performance.

\subsection{\textcolor{black}{Causal $\gamma$-DDPC}}
\textcolor{black}{Next we unveil a close connection between the causality-informed solution to the constrained least-squares problem \eqref{eq: causal K} and LQ factorization \eqref{eq: LQ factorization}, which, to the best of our knowledge, is a new result in both DDPC and subspace identification literature. }
\begin{lemma}\label{lemma1}
    The $i$th block row of $K^*$ in \eqref{eq: causal K}, denoted by $K^*_i$, can be compactly expressed as:
    \begin{align}
        &~ K^*_i \notag\\
        = & \left [ [L_{31,i} ~~ L_{32,i}^{\prime}] \begin{bmatrix}
            L_{11} & 0 \notag \\
            L_{21,[1:i]} & L_{22,[1:i],[1:i]}
        \end{bmatrix}^{-1} ~ \left | ~ 0_{p\times(L_f-i)m} \right . \right ] \notag \\
        = & ~[L_{31,i} ~~ L_{32,i}^{\prime} ~~ 0_{p\times (L_f-i)m}] \begin{bmatrix}
            L_{11} & 0 \\
            L_{21} & L_{22}
        \end{bmatrix}^{-1},\label{eq: Ki star}
        \end{align} 
where $L_{32,i}^{\prime} = L_{32,i}(:,1:im)$ indicates the first $i$ column blocks in the $i$th row block of $L_{32}$, and $L_{22,[1:i],[1:i]} = L_{22}(1:im,1:im)$ is the submatrix of $L_{22}$. As such, the optimal solution $K^*$ to \eqref{eq: causal K} can be succinctly expressed as:
\begin{equation}
    K^* = [L_{31} ~~ \mathcal{LT}_{p,m}(L_{32})] \begin{bmatrix}
            L_{11} & 0 \\
            L_{21} & L_{22}
        \end{bmatrix}^{-1}.
        \label{eq: K star}
\end{equation}
\end{lemma}
\begin{proof}
Let us define
\begin{equation}
U_{f,[1:i]} = \begin{bmatrix}
        u_{d,[L_p+1:L_p+N_d-L+2]} \\
        \vdots \\
        u_{d,[L_p+i:L_p+N_d-L+i+1]}
\end{bmatrix},
\end{equation}
which is composed of the first $i$ block rows of $U_f$. Meanwhile, we define the $i$th block row of $Y_{f}$ as $Y_{f,i} = y_{d,[L_p+i:L_p+N_d-L+i+1]}$. Note that each row block of $K^*$ can be independently solved for in \eqref{eq: causal K}. More precisely, the last $(L_f-i)$ column blocks of $K^*_i$ are known to be zero that strictly assure causality, i.e. $K_i^* = [\bar{K}_i^*~~0_{p\times(L_f-i)m}]$, where
\begin{equation*}
\begin{split}
    \bar{K}_i^* & = \arg \min_{\bar{K}_i} \left \| Y_{f,i} - \bar{K}_i \cdot \begin{bmatrix}
        Z_p \\ U_{f,[1:i]}
    \end{bmatrix} \right \|_F^2 = Y_{f,i} \begin{bmatrix}
        Z_p \\ U_{f,[1:i]}
    \end{bmatrix}^{\dagger} \\
    & = [L_{31,i} ~~ L_{32,i}^{\prime} ~~ 0_{p\times (L_f-i)m}] \begin{bmatrix}
            L_{11} & 0 \\
            L_{21} & L_{22}
        \end{bmatrix}^{-1}
\end{split}
\end{equation*}
is the solution to the constrained least-squares problem based on the LQ factorization as well as the triangular structure of $L_{22}$. Stacking \eqref{eq: Ki star} eventually yields \eqref{eq: K star}.
\end{proof}

\begin{remark}\label{remark:causal}
In the generic non-causal SPC predictor \eqref{eq: LS fit}, the fitting residual of least-squares is given by $\tilde{Y}_f = L_{33}Q_3$. When using the causal multi-step predictor $K^*$ in \eqref{eq: causal K}, the residual of constrained least-squares fitting is given by: $\tilde{Y}_f = Y_f - K^* \cdot \mathbf{col}(Z_p,U_f) = L^\prime_{32}Q_2 + L_{33}Q_3$,
where $L^\prime_{32} = L_{32} - \mathcal{LT}_{p,m}(L_{32})$ encloses non-causal components in $L_{32}$. 
It then directly follows that $\left \| L^\prime_{32}Q_2 + L_{33}Q_3 \right \|_F^2 \geq \left \| L_{33}Q_3 \right \|_F^2,$
where the equality holds if and only if $L^\prime_{32} = 0$. \textcolor{black}{Besides, the non-causal multi-step predictor $K^*$ in \eqref{eq: LS fit} has $pL_f\times[(m+p)L_p+mL_f]$ free parameters in total, which are \( pm L_f(L_f-1)/2 \) more than those in the causal predictor \eqref{eq: causal K}}. 
Henceforth, due to a smaller fitting residual \textcolor{black}{and more free parameters,} the non-causal multi-step predictor tends to embody more noise components than the causal one. This indicates that using the causal predictor can better hedge against data overfitting and suppress out-of-sample error, which can be beneficial for the control performance. 
\end{remark}

Lemma \ref{lemma1} indicates that LQ factorization \eqref{eq: LQ factorization} implies an \textcolor{black}{simple} solution to the constrained least-squares problem \eqref{eq: causal K}, which we are unaware of in literature. Based on this, we derive a new causal $\gamma$-DDPC (C-$\gamma$-DDPC) and its  equivalence with the C-SPC \eqref{eq: multi-step predictive control}. 
\begin{theorem} [Equivalence between C-$\gamma$-DDPC and C-SPC] \label{them-ca}
Solving the following data-driven control problem 
\begin{equation}
    \begin{aligned}
    \hspace{-20pt}\min_{u_f,\hat{y}_f,\gamma}& \mathcal{J}(u_f,\hat{y}_f) \\
        {\rm s.t.}~& \begin{bmatrix}
        L_{11} & 0& 0 & 0\\
            L_{21} & L_{22} & 0 & 0\\
            L_{31} & \mathcal{LT}_{p,m}(L_{32}) &L^\prime_{32}& L_{33}
        \end{bmatrix} \begin{bmatrix} 
        \gamma_1 \\ \gamma_2 \\ \gamma'_{2} \\ \gamma_3
        \end{bmatrix} = \begin{bmatrix} 
        z_p\\u_f \\ \hat{y}_f 
        \end{bmatrix}
       \\
        &~ ~ \gamma_{2}^\prime = 0, ~ \gamma_{3} = 0,~ u_f \in \mathbb{U},~ \hat{y}_f \in \mathbb{Y}  
    \end{aligned}\label{eq:30-enforced}
\end{equation}
yields the same control decision $u_f^*$ and output prediction $\hat{y}_f^*$ as those by solving the bilevel problem \eqref{eq: multi-step predictive control}.
\end{theorem}
\begin{proof} Defining 
\begin{equation}\label{Defi}
    \begin{bmatrix}
        \gamma_p \\
        \gamma_f \\
    \end{bmatrix} = \begin{bmatrix}
            L_{11} & 0 \\
            L_{21} & L_{22}
        \end{bmatrix}^{-1} \begin{bmatrix}
            z_p \\
            u_f
        \end{bmatrix},
\end{equation}
it follows that  
\begin{equation}
\gamma_p = L_{11}^{-1}z_p,~ u_f = L_{21}L_{11}^{-1}z_p + L_{22}\gamma_f.
    \label{eq: u_f}
\end{equation}
In virtue of Lemma \ref{lemma1}, the predictor in \eqref{eq: multi-step predictive control} can be equivalently expressed as:
    \begin{equation}
        \hat{y}_f = [L_{31} ~~ \mathcal{LT}_{p,m}(L_{32})] \begin{bmatrix}
            L_{11} & 0 \\
            L_{21} & L_{22}
        \end{bmatrix}^{-1} \begin{bmatrix}
            z_p \\
            u_f
        \end{bmatrix}.
        \label{eq: y_fhat}
    \end{equation}
Combining \eqref{Defi}-\eqref{eq: y_fhat} as a substitution of the multi-step predictor in \eqref{eq: multi-step predictive control} yields the desired result.
\end{proof}

Recall that \textcolor{black}{in Theorem \ref{thm: standard SPC}}, the classical SPC problem \eqref{eq: DeePC0p} can be recast into the LQ factorization-based \textcolor{black}{$\gamma$-DDPC} with $\Phi(\gamma)= \| \gamma_3 \|^2$ and $\mu \to \infty$. Notably, as compared to \textcolor{black}{$\gamma$-DDPC} \cite[\mbox{Eq.}~(51)]{breschi2023data}, the new causality-informed version of \textcolor{black}{C-$\gamma$-DDPC} in \eqref{eq:30-enforced} only calls for an additional block-triangularization of $L_{32}$, which can be interpreted as removing the ``non-causal" entries in order to enforce causality of multi-step prediction. In this sense, the proposed \textcolor{black}{C-$\gamma$-DDPC} in \eqref{eq:30-enforced} is as easy to solve as the generic \textcolor{black}{$\gamma$-DDPC}.

\textcolor{black}{The receding horizon implementation of \textcolor{black}{C-$\gamma$-DDPC} is outlined in Algorithm 1.} \textcolor{black}{Indeed, $\gamma_1^* = L_{11}^{-1} z_p$ can be computed prior to the receding horizon implementation. Under constraints $\gamma_{2}^\prime = 0$ and $\gamma_3 = 0$, the residuals of identifying the implicit causal multi-step predictor are ignored, while $\gamma_2$ acts as the only decision variable characterizing the input-output relation between $u_f$ and $\hat{y}_f$ for future prediction. The reformulation \eqref{eq:30-enforced} can be computationally advantageous, especially when adopted in a receding horizon fashion. More precisely, the dimension of decision variables $\gamma_2 \in \mathbb{R}^{(m+n)L_f}$ in \eqref{eq:30-enforced} is likely to be much smaller than that of $g \in \mathbb{R}^{N_d - L + 1}$ in \eqref{equation: DeePC division}, especially when $N_d$ is large.}

\begin{algorithm}[htbp]
	\caption{Receding horizon implementation of \textcolor{black}{C-$\gamma$-DDPC}}\label{alg1}	
	\begin{tabular}{ll}
		\textbf{Inputs:} 
	& Data $\{U_d,Y_d\}$; Trajectory  $z_{p}=\{u_p, y_p\}$; \\
        & Reference $r(t)$;  Matrices $\mathcal Q\succeq 0$ and $\mathcal R\succeq 0$; \\
        & Constraint sets $\mathbb{U}$ and $\mathbb{Y}$; \\
	\end{tabular}
        \algblock{Repeat}{EndRepeat}
        \algblockdefx{Repeat}{EndRepeat}{\textbf{repeat}}{\textbf{end}\ }
	\begin{algorithmic}[1]
         \State Compute $\{ L_{ij} \}_{i=1}^3,j=1,2,3$ via LQ factorization \eqref{eq: LQ factorization};
        \State Obtain $\gamma_1^* = L_{11}^{-1} z_p$; 
  	\Repeat
	\State Solve problem \eqref{eq:30-enforced} for $\{u^*_f(t)\}$;
	\State Apply the first optimal input in $\{u_f^*(t)\}$;
	\EndRepeat 
	\end{algorithmic}
\end{algorithm}

\textcolor{black}{
\begin{remark}\label{choselp}
In Theorem \ref{Fundamental Lemma}, it is required that $L_{\rm p}\ge\ell(A,B,C,D)$ for DDPC of deterministic systems \cite{markovsky2008data}. However, this may not be suitable in handling stochastic systems, because it is necessary to choose $L_{\rm p}$ sufficiently large to ensure $\| (A-KC)^{L_{\rm p}}\|_F\approx0$ \cite{chiuso2007role}, where $K$ is the steady state Kalman gain of \eqref{equation: LTI system}. 
 \end{remark}}


\begin{remark}[Comparison with segmented DDPC (S-DDPC) \cite{o2022data}]
    In \cite{o2022data}, S-DDPC was developed  by segmenting input-output trajectories within the control horizon $L_f$, which enables to partially alleviate the non-causality of DDPC. However, the induced implicit predictor is not completely causal. In addition, it requires $L_f$ to be a multiple of $L_p$, which is more restrictive than the proposed \textcolor{black}{C-$\gamma$-DDPC}.
\end{remark}

\subsection{Regularized causal $\gamma$-DDPC}
\textcolor{black}{The causal reformulation \eqref{eq:30-enforced} can also be relaxed using regularization to flexibly regulate the bias-variance trade-off. } We propose to drop two equalities in \eqref{eq:30-enforced} and regularize both $\gamma'_2$ and $\gamma_3$:
\begin{equation}\label{equation:regularized}
    \begin{aligned}
       \hspace{-10pt}\min_{u_f,\hat{y}_f,\gamma} & ~ \mathcal{J}(u_f,\hat{y}_f) +\lambda  \cdot \| \gamma'_{2} \|^2+ \mu  \cdot \| \gamma_3\|^2 \\
        {\rm s.t.}~& \begin{bmatrix}
        L_{11} & 0& 0 & 0\\
            L_{21} & L_{22} & 0 & 0\\
            L_{31} & \mathcal{LT}_{p,m}(L_{32}) & L_{32}^{\prime} & L_{33}
        \end{bmatrix} \begin{bmatrix} 
        \gamma_1 \\ \gamma_2 \\ \gamma'_{2} \\ \gamma_3
        \end{bmatrix} = \begin{bmatrix} 
        u_f \\ \hat{y}_f 
        \end{bmatrix} \\
        &~ u_f \in \mathbb{U},~ \hat{y}_f \in \mathbb{Y}
    \end{aligned}
\end{equation}
where $\lambda \ge 0$ and $\mu \ge 0$ are tuning parameters. Here, a hybrid regularizer $\lambda \cdot \| \gamma'_{2} \|^2 + \mu  \cdot \| \gamma_3\|^2$ yields a natural extension of the regularized non-causal formulation \eqref{eq: r-gamma DeePC} to the causality-informed setup. Notably, when solving \eqref{equation:regularized} with finite $\lambda$ and $\mu$, the implicit multi-step predictor no longer adheres strictly to the causal predictor identified by \eqref{eq: causal K}, which sacrifices some prediction accuracy for possible improvement of control performance. \textcolor{black}{To fine-tune regularization parameters in DDPC, the trial-and-error approach based on closed-loop experiments has been commonly utilized  \cite{breschi2023tuning}. More recently, some novel guidelines have been presented by \cite{breschi2023uncertaintyaware,breschi2023impact}, which enable to tune weights in regularized $\gamma$-DDPC prior to its deployment while circumventing closed-loop experiments. Since \eqref{equation:regularized} extends the generic regularized $\gamma$-DDPC, the guidelines in \cite{breschi2023uncertaintyaware,breschi2023impact} are also helpful for selecting $\mu$ and $\lambda$ in \eqref{equation:regularized}.}

\section{Numerical Case Studies}

\subsection{DDPC of stochastic LTI systems}
We consider a stochastic LTI system \eqref{equation: LTI system} where $w(t)$ and $v(t)$ are zero-mean white Gaussian noise. Its state-space model can be expressed in an innovation form:
\begin{equation*}
    \label{equation: system in innovation form}
    \left \{ 
    \begin{aligned}
        x(t+1)&=Ax(t)+Bu(t)+Ke(t)\\
        y(t)&=Cx(t)+Du(t)+e(t)
    \end{aligned}\right.
\end{equation*}
where the innovation $e(t)\sim\mathcal{N}(0,\sigma_e^2)$. The system matrices are given by
$
        A=\begin{bmatrix}
            0.7326&-0.0861\\0.1722&0.9909
        \end{bmatrix},
        ~B=\begin{bmatrix}
            0.0609\\0.0064
        \end{bmatrix},
        C=\begin{bmatrix}
            0&1.4142
        \end{bmatrix}, D=1,~K=[-0.3645 \quad 0.9973]^{\top}
$.
For offline data collection under open-loop condition, \textcolor{black}{a square wave with a period of $200$ time-steps and amplitude of $3$} is used as the persistently exciting input $\{u_d(i)\}_{i=1}^{N_d}$ in open-loop operations. We implement the following control strategies in a receding horizon fashion: (a) {C-$\gamma$-DDPC} \eqref{eq:30-enforced}; (b) {S-DDPC} \cite{o2022data}; (c) {$\gamma$-DDPC} \cite{breschi2023data}; (d) \textcolor{black}{{KF-MPC}, i.e. the oracle MPC with perfectly known $(A,B,C,D,K)$ and a Kalman filter implemented for online state estimation; (e) MOESP-KF-MPC, i.e. the indirect MPC with $(A,B,C,D,K)$ first identified from offline data using the Multivariable Output Error State-Space (MOESP) algorithm \cite{verhaegen1992subspace} and a Kalman filter implemented for online state estimation. Note that the latter two model-based methods inherently ensure causality of prediction through state-space equations.} \textcolor{black}{Following Remark \ref{choselp}, we set $L_p=15$ in DDPC, which ensures $\|(A-KC)^{L_{\rm p}}\|_2\le0.005$}, and $L_f=30$, $\mathcal{Q} = I_{p\times L_f}$, $\mathcal{R} = 0.05I_{m\times L_f}$, $\mathbb{U}=\{u(t)|-2\le u(t)\le2\}$, $\mathbb{Y}=\{\hat{y}(t)|-2\le\hat{y}(t)\le2\}$, $r(t)=\mathrm{sin}(2\pi t/N_c),\ t=1,2,\cdots,N_c$, where $N_c=60$ is the simulation duration. 
To evaluate the control performance, the index $ J=\sum_{t=1}^{N_c}||y(t) - r(t)||_\mathcal{Q}^2 + \sum_{t=1}^{N_c}||u(t)||_\mathcal{R}^2$ is used.
\begin{figure}[htbp]
    \centering

    \subfigure[$N_d$ = $200$]{
        \includegraphics[width=0.88\linewidth,clip,trim= 0 8 0 25]{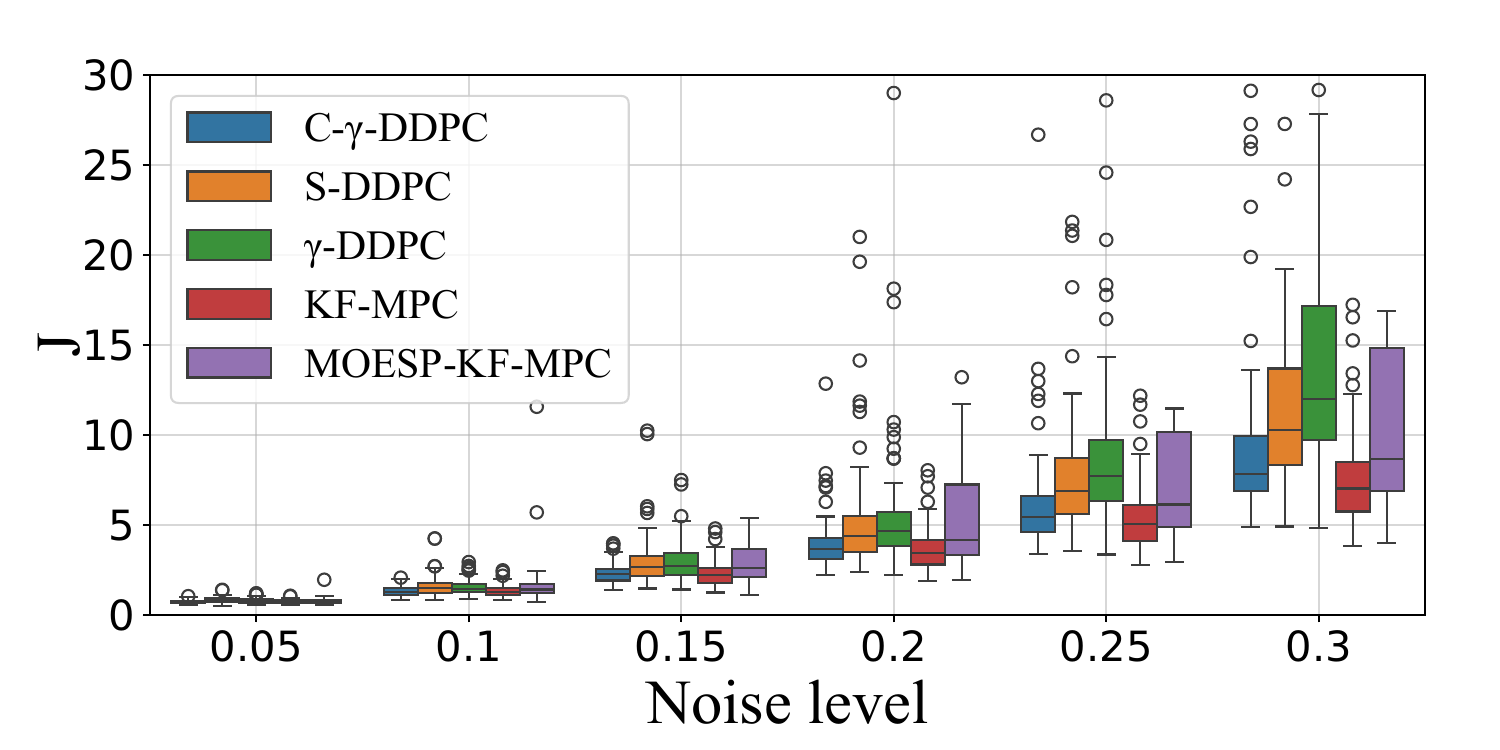}}

 \subfigure[$N_d$ = $400$]{
         \includegraphics[width=0.88\linewidth,clip,trim= 0 8 0 25]{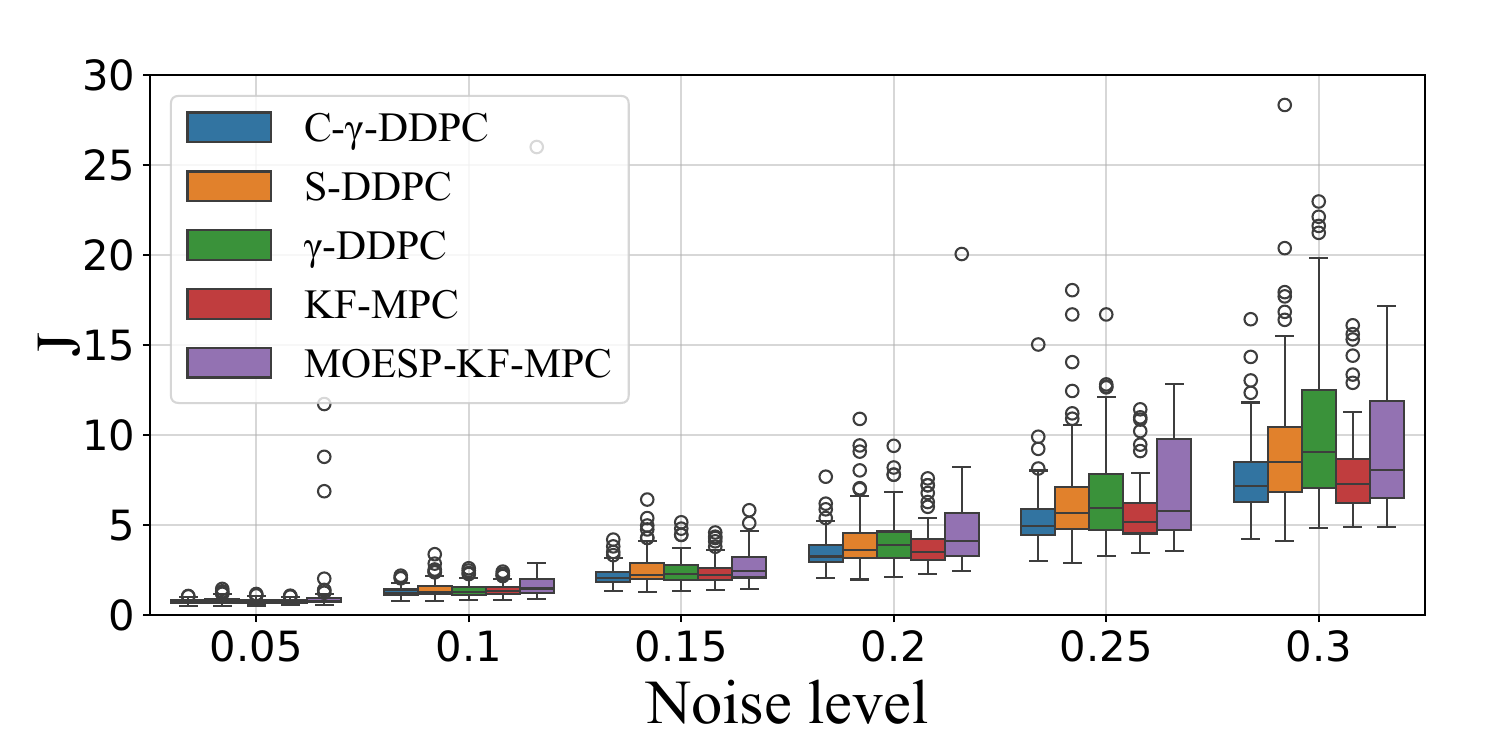}}

    \subfigure[$N_d$ = $600$]{
       \includegraphics[width=0.88\linewidth,clip,trim= 0 8 0 25]{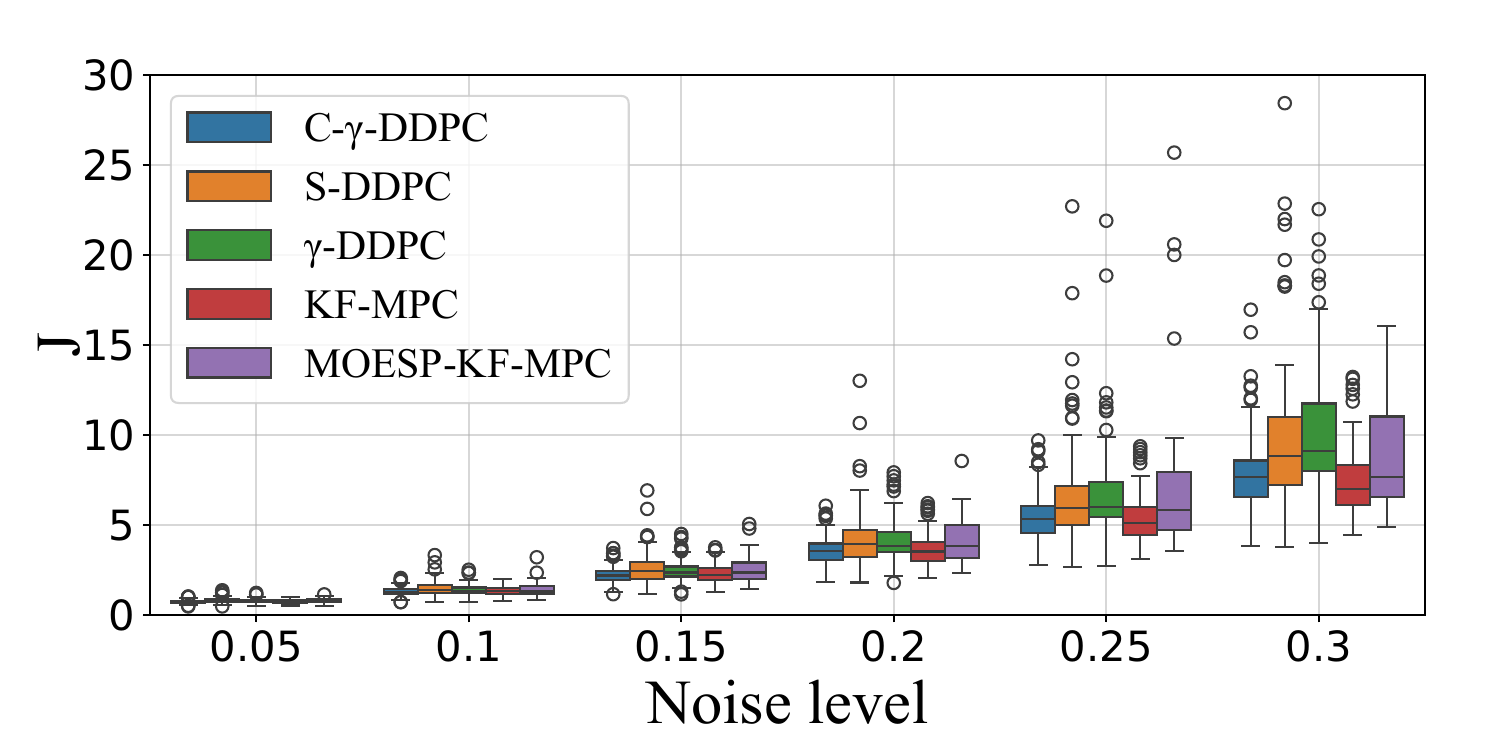}}

    \caption{Control performance of different approaches of LTI system with open-loop data collection in $100$ Monte Carlo simulations.}
    \label{fig:open-prediction performance}
\end{figure}

We create a variety of cases with different dataset sizes and noise levels, by letting $N_d=200$, $400$, and $600$, and varying $\sigma_e$ from $0.05$ to $0.3$ in increments of $0.05$. In each case, $100$ Monte Carlo runs are carried out to comprehensively evaluate the closed-loop control performance, and the results are shown in Fig. \ref{fig:open-prediction performance}. \textcolor{black}{It can be seen that when $\sigma_e=0$, three versions of DDPC exhibit identical performance because the multi-step predictors implicitly identified are perfectly causal. As the noise becomes heavier, the performance of DDPC gets negatively impacted, whereas the proposed C-$\gamma$-DDPC shows the best performance and the narrowest gap with the oracle KF-MPC. When the noise level is higher and $N_d$ is smaller, the outperformance of C-$\gamma$-DDPC becomes more evident. This is because for S-DDPC and $\gamma$-DDPC without fully ensuring causality, a heavier overfitting will occur for larger $\sigma_e$ and smaller $N_d$, which can be effectively alleviated by C-$\gamma$-DDPC. Meanwhile, without having to identify $(A,B,C,D,K)$ and run a Kalman filter, C-$\gamma$-DDPC also outperforms the indirect MOESP-KF-MPC in most cases. Thus, the causality-informed DDPC is attractive because of its easier implementation and better performance.} \textcolor{black}{We also investigate the computational cost. Under $ N_d = 200 $ and $ \sigma_e = 0.35 $, the average solution times for C-$\gamma$-DDPC, $\gamma$-DDPC, and KF-MPC are 0.85s, 0.87s, and 1.23s, respectively,\footnote{We solve all quadratic programs using OSQP \cite{stellato2020osqp} on a desktop computer with an Intel Core(TM) i7-12700F CPU and 16 GB RAM.} indicating that $\gamma$-DDPC and its causal version enjoy computational efficiency.}




Next, we delve into the performance of three regularized DDPC methods, i.e. {\textcolor{black}{RC-$\gamma$-DDPC} \eqref{equation:regularized}}, RS-DDPC \cite{o2022data} and {\textcolor{black}{R-$\gamma$-DDPC \cite{breschi2023uncertaintyaware}}}, which are relaxations of {\textcolor{black}{C-$\gamma$-DDPC}}, {\textcolor{black}{S-DDPC}} and {\textcolor{black}{$\gamma$-DDPC}}, respectively. \textcolor{black}{The regularization parameters are optimally selected from a logarithmically spaced grid of $100$ points within $[10^{-5},10^{5}]$}. We carry out 100 Monte Carlo simulations under $N_d = 200,400,600$ and $\sigma_e = 0.35$. For a clear comparison, we normalize control costs of various methods such that the control cost of \textcolor{black}{RC-$\gamma$-DDPC} is 1. The results are outlined in Table \ref{tab:nonlinear-open-prediction performance}, where the results of \textcolor{black}{C-$\gamma$-DDPC} without regularization are also reported. It can be seen \textcolor{black}{RC-$\gamma$-DDPC} performs invariably the best under all circumstances.  As $N_d$ increases, the advantage of \textcolor{black}{RC-$\gamma$-DDPC} over \textcolor{black}{RS-DDPC} and \textcolor{black}{R-$\gamma$-DDPC} gradually vanishes, mainly because the lack of causality in the latter two methods can be better alleviated using more data. Meanwhile, \textcolor{black}{RC-$\gamma$-DDPC} outperforms \textcolor{black}{C-$\gamma$-DDPC}, and their gap becomes smaller with $N_d$ increasing. This showcases the usefulness of relaxation in the absence of abundant data. 
\begin{table}[htb]
    \centering
    \caption{Control Costs of LTI and Nonlinear Systems}
    \label{tab:nonlinear-open-prediction performance}
    \begin{tabular}{ccccc}
  \toprule
  & &{$N_d=200$}&{$N_d=400$}&{$N_d=600$}\\
    \midrule
   \multirow{4}{*}{\makecell[c]{LTI\\System\\($\sigma_e=0.35$)} }
&   \textcolor{black}{RC-$\gamma$-DDPC}&1&1&1\\
&    \textcolor{black}{RS-DDPC}&1.1689&1.0970&1.0508\\
&    \textcolor{black}{R-$\gamma$-DDPC}&1.3140&1.1190&1.0933\\
&    \textcolor{black}{C-$\gamma$-DDPC}&1.0581&1.0460&1.0214\\
   \midrule
  \multirow{4}{*}{\makecell[c]{Nonlinear\\System\\($\sigma_e=0.3$)} }
&    \textcolor{black}{RC-$\gamma$-DDPC}&1&1&1\\
&    RS-DDPC&1.2094&1.0893&1.1024\\
&    R-$\gamma$-DDPC&1.4660&1.1674&1.1360\\
&    \textcolor{black}{C-$\gamma$-DDPC}&1.1113&1.0828&1.0350\\
  \bottomrule
    \end{tabular}
\end{table}

\subsection{DDPC of nonlinear systems}
Real-world dynamical systems are inevitably subject to nonlinearity. Even if DDPC uses an implicit linear predictor, its effectiveness in tackling nonlinear systems has been validated by a number of empirical findings even if a wrong model class is selected \cite{huang2021quadratic,huang2021robust}. Now, we examine the efficacy of the proposed causality-informed formulations of DDPC through the following nonlinear system:
\begin{equation*}
    \label{equation: nonlinear system}
    \left \{
     \begin{aligned}
    x(t+1)&=A\Tilde x(t)+B\Tilde{u}(t)+Ke(t)\\
    y(t)&=Cx(t)+D\Tilde{u}(t)+e(t)
    \end{aligned}\right.
\end{equation*}
\textcolor{black}{where $\Tilde{x}(t)=(1-\varepsilon)x(t)+0.5\varepsilon x^3(t)$ and $\Tilde{u}(t)=(1-\varepsilon)u(t)+\varepsilon[{\rm sin}(u(t))+2u^3(t)]$ are monotonically increasing around zero, which is an equilibrium point. Thus, a linearized model yields a good approximation of the system near the equilibrium.} The degree of system nonlinearity is characterized by $\varepsilon \in [0,1]$. 
When $\varepsilon = 0$, the model dynamics is entirely linear, and the level of nonlinearity increases with $\varepsilon$.
\textcolor{black}{The reference $r(t)$ and system matrices remain the same as the previous subsection.} 

\begin{figure}[!t]
     \centering
 \subfigure[$\sigma_e=0$]{
\includegraphics[width=0.38\textwidth,clip,trim= 0 12 18 35]{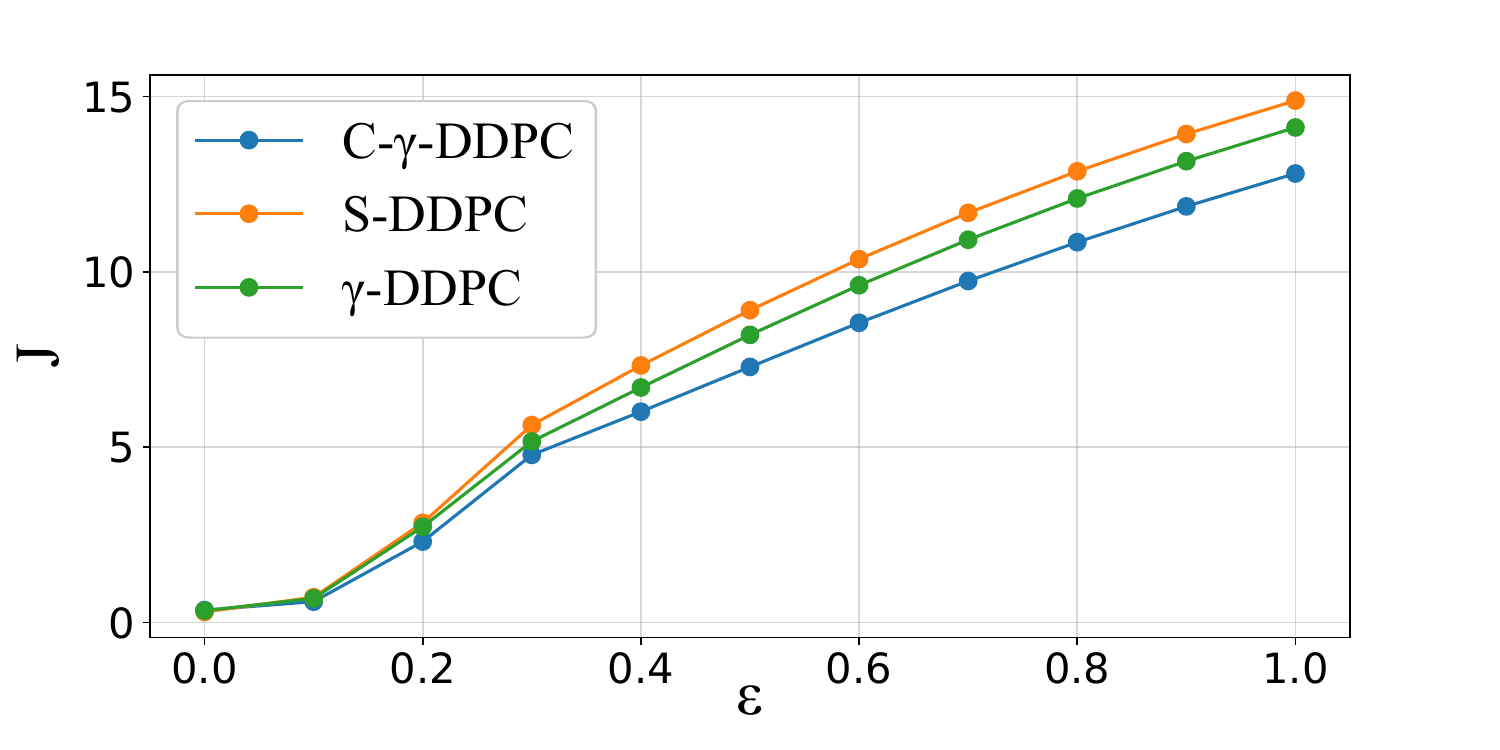}}
 \label{fig:a-nonlinear-open-C-SPC, S-SPC, and SPC}
     \subfigure[$\sigma_e=0.05$]{
       \includegraphics[width=0.36\textwidth,clip,trim= 10 12 0 35]{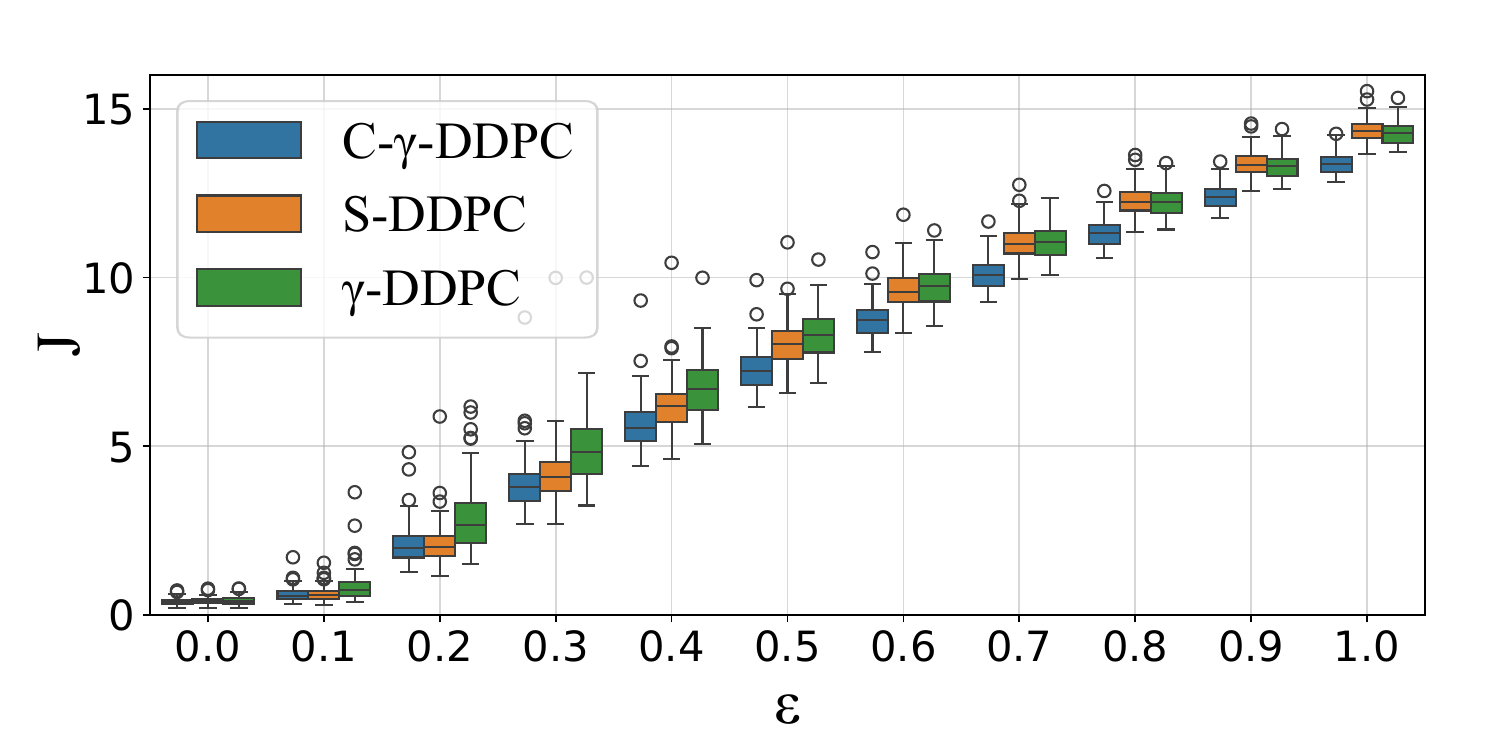}}
     \label{fig:b-nonlinear-open-C-SPC, S-SPC, and SPC}
             \caption{Control performance of different approaches of nonlinear system with open-loop data collection ($N_d=200$). 
             }
    \label{fig:nonlinear-open-C-SPC, S-SPC, and SPC}
\end{figure}

We run the nonlinear system around the equilibrium and collect input-output trajectories of length $N_d = 200$ from open-loop condition to implement different DDPC strategies. We first vary $\varepsilon$ between $0$ and $1$ in the noise-free case with $\sigma_e=0$ such that the system becomes deterministic. The results of \textcolor{black}{C-$\gamma$-DDPC}, \textcolor{black}{S-DDPC}, and \textcolor{black}{$\gamma$-DDPC} are depicted in Fig. \ref{fig:nonlinear-open-C-SPC, S-SPC, and SPC}(a). It can be seen that the performance of all methods degrades with $\varepsilon$ increasing, while \textcolor{black}{C-$\gamma$-DDPC} outperforms the other methods by a large margin. Interestingly, despite the absence of uncertainty, considering causality in the implicit input-output predictor still brings obvious benefits to the realized control performance. This is plausible because when a linear model class is used to characterize nonlinear dynamics in \textcolor{black}{$\gamma$-DDPC} and \textcolor{black}{S-DDPC}, the implicit non-causal predictor will inevitably enclose model mismatch components in the non-causal terms, thereby leading to degraded prediction accuracy. The results of 100 Monte Carlo runs in the noisy setting ($\sigma_e=0.05$) are further shown in Fig. \ref{fig:nonlinear-open-C-SPC, S-SPC, and SPC}(b), where \textcolor{black}{C-$\gamma$-DDPC} consistently outperforms the other two methods. This showcases that considering causality is not only useful in reducing errors in the presence of stochastic uncertainty but also important in handling nonlinear systems. The performance metrics of three regularized versions of DDPC that are optimally tuned are also presented in Table \ref{tab:nonlinear-open-prediction performance}, where the outperformance of \textcolor{black}{RC-$\gamma$-DDPC} is evident.

\section{Predictive Control of A Simulated Industrial Heating Furnace}
 As a key equipment in petro-chemical industry, the tubular furnace has been widely used for heating crude oil to a desired temperature before feeding into downstream units. The furnace system has two inputs, i.e. flow rates of natural gas ($u_1$, \SI{}{m^3/h}) and air ($u_2$, \SI{}{m^3/h}), and has two outputs, i.e. outlet temperature of crude oil ($y_1$, \SI{}{\degreeCelsius}) and ${\rm O}_2$ content of stack gas ($y_2$, \SI{}{\%}), as depicted in Fig. \ref{fig: furnace}. \textcolor{black}{The combustion process within the furnace has nonlinearity in both its inputs and states because the flow rates of fuel gas and air as two inputs, and the ${\rm O}_2$ content in the furnace as a state variable are known to have a nonlinear impact on the combustion efficiency.} High-fidelity simulations of this furnace is enabled by the Fired Process Heater (FPH) simulator in the Honeywell UniSim Design Suite.

\begin{figure}[h]
    \centering
    \includegraphics[width=0.4\textwidth]{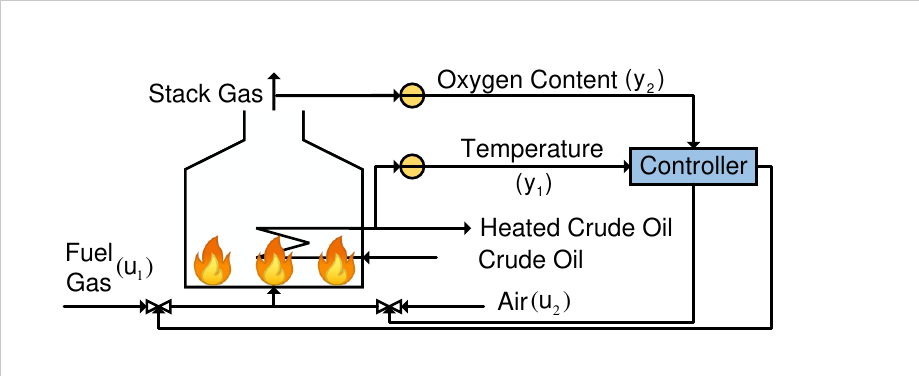}
    \caption{Structural diagram of industrial tubular furnace}
    \label{fig: furnace}
\end{figure}

\subsection{Data collection under open-loop conditions}
Considering the slow dynamics of the process, data collection and control are carried out with an interval of 3 mins. \textcolor{black}{To implement data-driven control strategies, a dataset $\{u_d(i), y_d(i)\}_{i=1}^{N_d}$ of size $N_d=3000$ is collected by varying $u_1$ between $\SI{2740}{m^3/h}$ with $\SI{3160}{m^3/h}$ and $u_2$ between $\SI{30000}{m^3/h}$ with $\SI{32000}{m^3/h}$ in open-loop conditions.} \textcolor{black}{The inputs $u_1$ and $u_2$ are set as square waves with periods of 500 and 660 time-steps, respectively.} Data collected are then normalized to account for different scales of process variables. \textcolor{black}{In online receding horizon control, the goal is to make $y_1$ track a square wave varying between 360\SI{}{\degreeCelsius} and 361\SI{}{\degreeCelsius}}, while maintaining $y_2$ around $\SI{3}{\%}$ for high combustion efficiency. We choose $L_p=50$, $L_f=100$, $\mathcal Q={\rm diag}(1_{L_f}\otimes[10^{-2},5\times10^{-3}])$, and $\mathcal R={\rm diag}(1_{L_f}\otimes[10^{-4},10^{-4}])$. The constraint set $\mathbb{U}$ for inputs is designed such that $u_1\in[\SI{2100}{m^3/h},\SI{3200}{m^3/h}]$ and $u_2\in[\SI{24000}{m^3/h},\SI{34000}{m^3/h}]$.


The process outputs under different data-driven control strategies, including \textcolor{black}{$\gamma$-DDPC}, \textcolor{black}{S-DDPC}, \textcolor{black}{C-$\gamma$-DDPC} and \textcolor{black}{RC-$\gamma$-DDPC}, are shown in Fig. \ref{fig:Hysys Output}, and empirical performance metrics such as average tracking errors and control costs are detailed in Table \ref{tab:Hysys tracking performance}, where $\tilde y_1(t)=y_1(t)-r_1(t)$ and $\tilde y_2(t)=y_2(t)-r_2(t)$ denote output tracking errors. It can be observed that \textcolor{black}{$\gamma$-DDPC} yields oscillatory control behavior that is unsatisfactory as compared to \textcolor{black}{S-DDPC} and \textcolor{black}{C-$\gamma$-DDPC}, indicating that considering causality does improve the control performance. Meanwhile, the tracking performance of \textcolor{black}{C-$\gamma$-DDPC} is superior to that of \textcolor{black}{S-DDPC} due to a strictly causal predictor identified from data. The results in Table \ref{tab:Hysys tracking performance} further indicates that \textcolor{black}{C-$\gamma$-DDPC} and \textcolor{black}{RC-$\gamma$-DDPC} yield a desirable performance by accounting for causality. Besides, \textcolor{black}{RC-$\gamma$-DDPC} performs better than the \textcolor{black}{C-$\gamma$-DDPC}, which showcases the benefit of regularization in improving the control performance of this nonlinear process. 

\begin{figure}[htb!]
    \centering
    \subfigure[Outlet temperature of crude oil]{
        \includegraphics[width=0.42\textwidth]{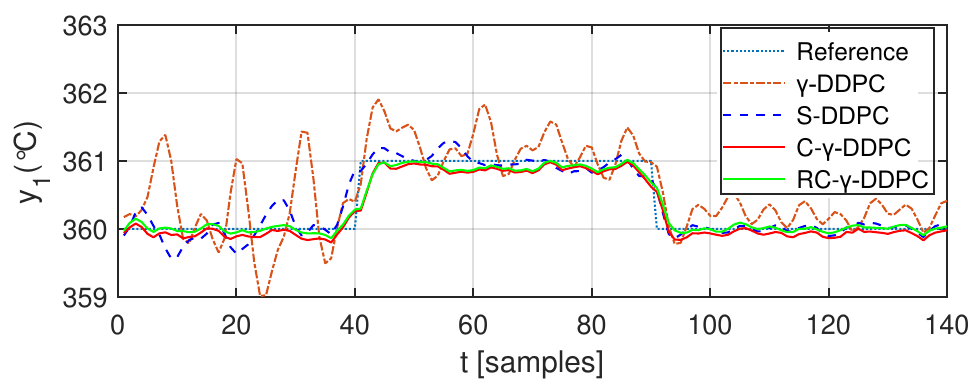}
    }
    \subfigure[${\rm O_2}$ content of stack gas]{
        \includegraphics[width=0.42\textwidth]{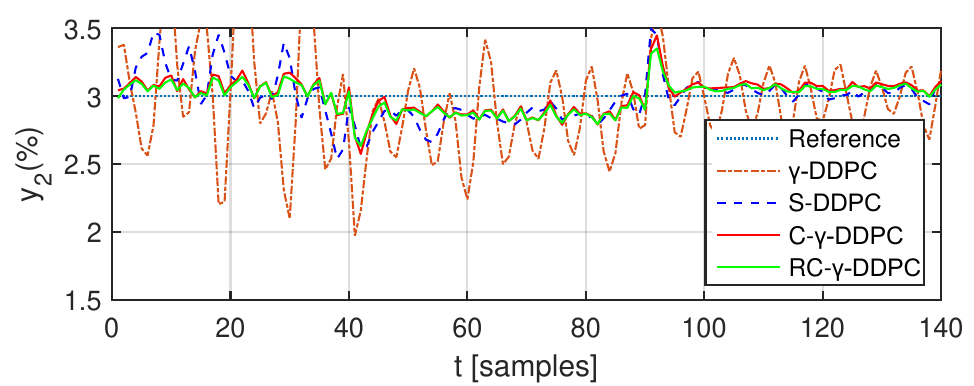}
    }
    \caption{Outputs of the tubular furnace system controlled by different controllers using open-loop data, including \textcolor{black}{$\gamma$-DDPC}, \textcolor{black}{S-DDPC}, the proposed \textcolor{black}{C-$\gamma$-DDPC} and \textcolor{black}{RC-$\gamma$-DDPC} ($\lambda=5\times10^{-3}$, $\mu=0.1$).}
    \label{fig:Hysys Output}
\end{figure}

\begin{table}[htb!]
    \centering
    \caption{Control Cost of Heating Furnace with Open-Loop Data}
    \label{tab:Hysys tracking performance}
    \tabcolsep=3pt
    \begin{tabular}{c c c c c}
  \toprule
   &$\mathbb{E} \{ \|\tilde{y}_1(t)\|_2^2 \}$&$\mathbb{E} \{  \|\tilde{y}_2(t)\|_2^2 \}$&$\mathbb{E} \{\|u_1(t)\|_2^2 \}$ & $\mathbb{E} \{ \|u_2(t)\|_2^2 \}$\\
   \midrule
    \textcolor{black}{$\gamma$-DDPC}&0.2076&0.6081&11.685&4.031\\
    \textcolor{black}{S-DDPC}&0.0318&0.1271&12.163&4.193\\
    \textcolor{black}{C-$\gamma$-DDPC}&0.0199&0.0611&12.615&4.449\\
    \textcolor{black}{RC-$\gamma$-DDPC}&0.0148&0.0517&12.293&4.329\\
  \bottomrule
    \end{tabular}
\end{table}

\subsection{Data collection under closed-loop conditions}
In industrial practice, one may not be allowed to operate the process in open-loop. A practically relevant scenario is that for safety or efficiency concerns, the process has to be operated in closed-loop conditions. 
We close the loop by a coarsely tuned stabilizing PI controller, \textcolor{black}{which has a state-space parameterization $(A_c,B_c,C_c,D_c)$}: 
       $ A_c=C_c=\begin{bmatrix}
           1 & 0 \\ 1 & 0
        \end{bmatrix},~
        B_c=\begin{bmatrix}
            0.08 & 0 \\ 0.29 & 0
        \end{bmatrix},~
        D_c=\begin{bmatrix}
           0.32 & 0 \\ 0.62 & 0
        \end{bmatrix}$. Under closed-loop feedback control, an input-output trajectory of length $N_d=2000$ is generated by varying the setpoint of $y_1$ between $\SI{360}{\degreeCelsius}$ and $\SI{366}{\degreeCelsius}$ and that of $y_2$ between $\SI{1.5}{\%}$ and $\SI{2.5}{\%}$. \textcolor{black}{This helps to collect input data that are persistently exciting; however, in stochastic systems, the correlation between inputs and past innovations may be an issue for DDPC \cite{pan2022stochastic}.} In online predictive control, the reference of $y_1$ is a square wave \textcolor{black}{ with a period of $50$ time-steps and amplitude of $1$} varying between $\SI{363}{\degreeCelsius}$ and $\SI{364}{\degreeCelsius}$, while a constant reference of $3\%$ is set for  $y_2$. For control design, we set $\mathcal Q={\rm diag}(1_{L_f}\otimes[10^{-3},5\times10^{-3}])$ and $\mathcal R={\rm diag}(1_{L_f}\otimes[10^{-4},10^{-4}])$, while the other parameters are the same as those in the previous open-loop case.

\begin{figure}[htbp]
    \centering
    \subfigure[Outlet temperature of crude oil]{
        \includegraphics[width=0.42\textwidth]{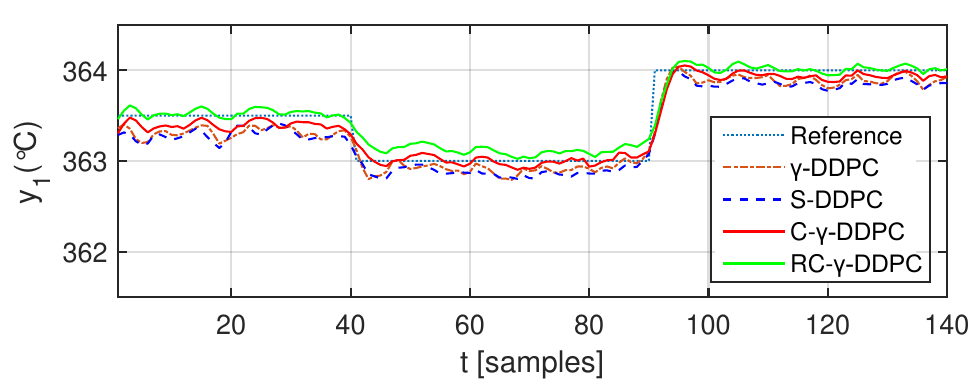}
    }
    \subfigure[${\rm O_2}$ content of stack gas]{
        \includegraphics[width=0.42\textwidth]{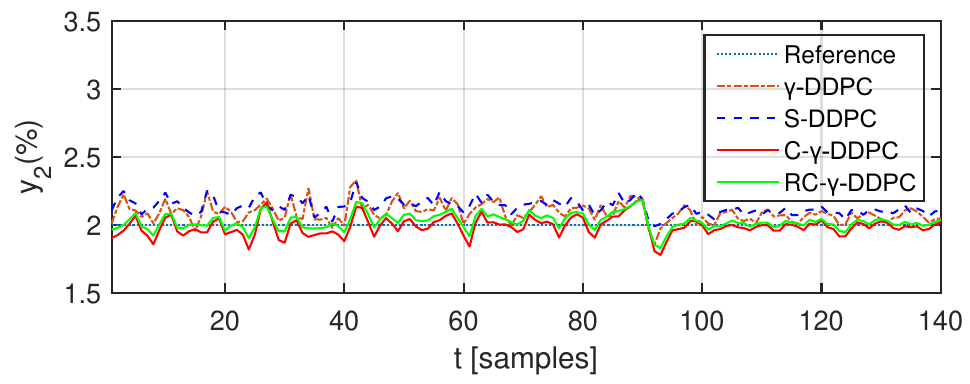}
    }
    \caption{Outputs of the tubular furnace system controlled by different controllers using closed-loop data, including \textcolor{black}{$\gamma$-DDPC}, \textcolor{black}{S-DDPC}, the proposed \textcolor{black}{C-$\gamma$-DDPC} and \textcolor{black}{RC-$\gamma$-DDPC}.}
    \label{fig:Hysys Output closed-loop}
\end{figure}

We implement \textcolor{black}{$\gamma$-DDPC}, \textcolor{black}{S-DDPC}, \textcolor{black}{C-$\gamma$-DDPC}, and \textcolor{black}{RC-$\gamma$-DDPC} in a receding horizon fashion, and investigate their realized performance. In \textcolor{black}{RC-$\gamma$-DDPC}, the regularization parameters are chosen as $\lambda=5\times10^{-3}$ and $\mu=0.1$. The controlled outputs are depicted in Fig. \ref{fig:Hysys Output closed-loop} and empirical performance metrics can be found in Table \ref{tab:Hysys tracking performance closed-loop}. It can be seen that \textcolor{black}{S-DDPC} exhibits an even less favorable tracking than \textcolor{black}{$\gamma$-DDPC}, whereas \textcolor{black}{C-$\gamma$-DDPC} exhibits much lower control errors than \textcolor{black}{$\gamma$-DDPC} due to the enforced causality. In addition, the regularized formulation \textcolor{black}{RC-$\gamma$-DDPC}, as a relaxation of \textcolor{black}{C-$\gamma$-DDPC}, gains an upper hand in terms of tracking errors and control cost. This showcases the usefulness of enforcing a causal predictive relation in DDPC, based on which suitable regularizaion may bring further benefit.
\begin{table}[htbp]
    \centering
    \caption{Control Cost of Heating Furnace with Closed-Loop Data}
    \label{tab:Hysys tracking performance closed-loop}
    \tabcolsep=3pt
    \begin{tabular}{c c c  c c}
  \toprule
    &$\mathbb{E} \{ \|\tilde{y}_1(t)\|_2^2 \}$&$\mathbb{E} \{ \|\tilde{y}_2(t)\|_2^2 \}$&$\mathbb{E} \{\|u_1(t)\|_2^2 \}$ & $\mathbb{E} \{ \|u_2(t)\|_2^2 \}$\\
   \midrule
   \textcolor{black}{$\gamma$-DDPC}&0.0235&0.0527&1.6463&1.4371\\
   \textcolor{black}{S-DDPC}&0.0325&0.0819&1.6930&1.3320\\
   \textcolor{black}{C-$\gamma$-DDPC}&0.0123&0.0183&1.5110&1.7562\\
      \textcolor{black}{RC-$\gamma$-DDPC}&0.0123&0.0148&1.3039&1.3625\\
  \bottomrule
    \end{tabular}
\end{table}

\section{Conclusion}

In this \textcolor{black}{brief}, a novel causality-informed formulation of DDPC was derived, which brings improved performance over the existing DDPC methodology that falls short of considering causality. \textcolor{black}{Using LQ factorization, a simple formulation of causal $\gamma$-DDPC was developed that only entails minor modification of generic non-causal $\gamma$-DDPC.} 
\textcolor{black}{A regularized form was further derived that can help improving the control performance.} As useful extensions of the generic $\gamma$-DDPC, the proposed causal formulations are as easy to solve as $\gamma$-DDPC and its regularized version.
Numerical examples and applications to a simulated furnace demonstrated the improved performance of the causality-informed DDPC, in the presence of not only stochastic noise but also process nonlinearity. 

\bibliographystyle{IEEEtran}
\bibliography{ref_bib}

\end{document}